\documentclass[12pt]{amsart}

\usepackage{upref,amssymb,bm,mathrsfs}
\usepackage{enumerate}
\usepackage{mathtools}
\usepackage{tikz}
\usepackage{comment}

\usepackage[centering,width=6in,truedimen]{geometry}
\usepackage[colorlinks,pagebackref,pdfstartview=FitH]{hyperref}

\hypersetup{
    colorlinks,
    citecolor=blue,
    filecolor=black,
    linkcolor=red,
    urlcolor=black
}

\theoremstyle{plain}
\newtheorem{theorem}{Theorem}[section]
\newtheorem{proposition}[theorem]{Proposition}
\newtheorem{lem}[theorem]{Lemma}

\newtheorem*{ques*}{Question}

\theoremstyle{definition}
\newtheorem{defn}[theorem]{Definition}

\theoremstyle{remark}

\DeclareMathOperator{\hdim}{dim_H}

\DeclareMathOperator{\pdim}{dim_P}
\DeclareMathOperator{\edim}{dim_e}

\newcommand{\R}{\mathbb R}

\numberwithin{equation}{section}

\title[A non-sticky Kakeya set of Lebesgue measure zero]{A non-sticky Kakeya set of Lebesgue measure zero}

\author{Chun-Kit Lai}
\address{
	Department of Mathematics\\
	San Francisco State University\\
	1600 Holloway Avenue, San Francisco, CA 94132
}
\email{cklai@sfsu.edu}

\author{Adeline E. Wong}
\address{
	Department of Mathematics\\
	San Francisco State University\\
	1600 Holloway Avenue, San Francisco, CA 94132
}
\email{awong47@sfsu.edu}

\begin{document}

\begin{abstract}
The Kakeya set conjecture in $\R^3$ was recently resolved by Wang and Zahl. The distinction between sticky and non-sticky Kakeya sets plays an important role in their proof. Although the proof did not require the Kakeya set to be Lebesgue measure zero, measure zero Kakeya sets are the crucial case whose study is required to resolve the conjecture. In this paper, we explicitly construct a non-sticky Kakeya set of Lebesgue measure zero in $\R^2$ (and hence in any dimension). We also construct non-trivial sticky and non-sticky Kakeya sets in high dimension that are not formed by taking the Cartesian product of a 2-dimensional Kakeya set with $\R^{d-2}$, and we verify that both Kakeya sets have Hausdorff dimension $d$. 
    
\end{abstract}

\maketitle
\section{Introduction}
A {\bf Kakeya set} (or a Bescovitch set) in $\R^d$ is a compact set containing a unit line segment in every direction. The {\it Kakeya set conjecture} states that every Kakeya set has Hausdorff dimension $d$. This conjecture is one of the most notoriously difficult, and significant, conjectures in harmonic analysis and geometric measure theory, and it is implied by several other important conjectures, such as the Fourier restriction conjecture and Bochner-Riesz conjecture in harmonic analysis. For more background, readers can consult the books \cite{Mattila_2015,W03}.

The conjecture can be proved fairly easily in $\R^2$, which Davies did in the 1970s \cite{Davis71} (see also \cite{Mattila_2015}). Recently, the conjecture was resolved in $\R^3$ by Wang and Zahl via a series of seminal works \cite{wang2022stickykakeyasetssticky,wang2025assouaddimensionkakeyasets,wang2025volumeestimatesunionsconvex}. One of the key ingredients in their proof was to distinguish two classes of Kakeya sets known as the sticky and the non-sticky Kakeya sets. The stickiness property was first formulated by Katz, {\L}aba, and Tao in early 2000 \cite{KLT2000}. In 2014, Tao laid down a strategy to attack the Kakeya conjecture in his blog \cite{Taoblog}; in summary, if a hypothetical counterexample to the Kakeya set conjecture existed, the set ought to be planey, grainy, and sticky. While planiness and graininess were well-characterized in \cite{BCT2006} and \cite{Guth2016} (see also \cite{KZ2019}), the correct formulation of stickiness was not clear at that time. Wang and Zahl advanced Tao's approach by defining an explicit notion of stickiness and showed first that a sticky Kakeya set in $\R^3$ must be of full Hausdorff dimension \cite{wang2022stickykakeyasetssticky}, before settling the conjecture on $\R^3$ in full generality \cite{wang2025volumeestimatesunionsconvex}. In \cite{wang2022stickykakeyasetssticky}, they also proposed the 

\smallskip

\noindent{\it Sticky Kakeya set conjecture:}  All sticky Kakeya sets in $\R^d$ must have Hausdorff dimension $d.$

\smallskip

\noindent Even in this restricted version, the conjecture is open for $d \ge  4$, with some partial results for $\R^4$ in \cite{choudhuri2024improvedboundhausdorffdimension}.

To describe stickiness precisely, we will need to define a space of affine lines. An affine line on $\R^d$, not lying in the hyperplane $\{(x_1,\cdots, x_d): x_d = 0\}$,  can be parametrized by 
\begin{equation}\label{eqL_va}
L_{{\bf v},{\bf a}}(t) = ({\bf a},0)+ t({\bf v}, 1), t\in\R, 
\end{equation}
where ${\bf a}, {\bf v}\in \R^{d-1}$. Let ${\mathcal A}$ be the collection of all such lines. Technically, we are missing the lines that lie on the hyperplane $x_d = 0$. However, the missing directions only form a lower-dimensional set in the space of directions, so there is no loss of generality to disregard this set of lines.   We now endow ${\mathcal A}$ with a metric on $\R^{2d-2}$ by 
$$d(L_{{\bf v},{\bf a}}, L_{{\bf w},{\bf b}}) = \|{\bf a}-{\bf b}\|+\|{\bf v}-{\bf w}\|.$$
Suppose that ${\mathbb K}$ is a Kakeya set in $\R^d$. Then for all ${\bf v}\in\R^{d-1}$, there exists ${\bf a}_{\bf v}$ such that $L_{{\bf v},{}\bf a_v}\cap {\mathbb K}$ contains a unit line segment. Let 
$$
{\mathcal A}_{\mathbb K} = \{({\bf v},{\bf a_v}): L_{{\bf v},{\bf a_v}}\cap {\mathbb K}  \ \mbox{contains a unit line segment}\}. 
$$
Because the set has full dimension in ${\bf v}$ by the definition of Kakeya set, the Hausdorff dimension of ${\mathcal A}_{\mathbb K}$ in $\R^{2d-2}$ is at least $d-1$. 
\begin{defn}
We say that a Kakeya set ${\mathbb K}$  is a {\bf sticky Kakeya set} if the packing dimension of  ${\mathcal A}_{\mathbb K}$ is equal to $ d-1$. Otherwise, ${\mathbb K}$  is a {\bf non-sticky Kakeya set} (i.e., the packing dimension of ${\mathcal A}_{\mathbb K}$ is strictly larger than $d-1$).
\end{defn}

As a simple example, the unit disk is a non-sticky Kakeya set, since each direction corresponds to an interval of initial points, making the set ${\mathcal A}_{\mathbb K}$ dimension 2, and indeed of positive Lebesgue measure. On the other hand, the classic example of a measure zero Kakeya set is the one based on the four-corner self-similar Cantor set $K$, which has Hausdorff dimension 1. This construction can be found in \cite[Chapter 7]{Stein-real} (see also Section \ref{section4}). In this construction, ${\mathcal A}_{\mathbb K} = K$, which also has packing dimension equal to $1$, hence $K$ is a sticky Kakeya set.   

Although a restriction to Kakeya sets of Lebesgue measure zero was not used in the proof of the Kakeya conjecture in $\R^3$, it is clear that we can, without loss of generality, assume the set has measure zero (since sets with positive measures have full Hausdorff dimension). As there appear to be no known examples of a non-sticky Kakeya set of Lebesgue measure zero, one might conjecture that non-sticky Kakeya sets must have positive Lebesgue measure, so that the sticky Kakeya set conjecture and the original Kakeya set conjecture are equivalent. In this paper, we show that this is not the case; the main result of this paper is

\begin{theorem}\label{main theorem}
    There exists a non-sticky Kakeya set of Lebesgue measure zero in $\R^2$, and hence in $\R^d$, for all $d>2$.  
\end{theorem}

A construction for $\R^d$ with $d>2$ can be made by simply taking the Cartesian product of a two-dimensional Kakeya set with $\R^{d-2}$.  In the literature, apart from this trivial construction, there appear to be no other explicit examples of measure zero Kakeya sets in high dimension. It would be more interesting to construct additional genuine examples of Kakeya sets of measure zero, in the sense that they are not formed by this Cartesian product procedure. In Section \ref{section4}, we provide simple constructions of such Kakeya sets for both the sticky and non-sticky cases. Moreover, we show that they indeed fulfill the Kakeya set conjecture.

In Section \ref{section2}, we will present the necessary tools for our construction. In Section \ref{section3}, we prove Theorem \ref{main theorem}. In Section \ref{section4}, we provide the non-trivial Kakeya sets mentioned in the previous paragraph.


\section{Preliminaries}\label{section2}
This section will present the basic tools for our construction. Throughout our paper, ${\mathcal H}^{\alpha}$ denotes the $\alpha$-Hausdorff measure; $\hdim$ and $\pdim$ denote the Hausdorff and packing dimensions, whose definitions can be found in \cite{BKS2024}, of a set; and $m^d(E)$ denotes the $d$-dimensional Lebesgue measure of a measurable set $E$. 

\subsection{General construction of Kakeya sets.} We can produce a Kakeya set in $\R^d$ by specifying a compact set $K\subset \R^{2d-2}$ to be ${\mathcal A}_{\mathbb K}$. To be precise, for each $x\in K$, let us write $x = ({\bf v},{\bf a})$, where ${\bf v,a}\in \R^{d-1}$, and $L_{{\bf v},{\bf a}}$ for the line defined in (\ref{eqL_va}) with $t\in[0,1]$. 

\begin{proposition}\label{prop0}
    Let $K$ be a compact set in $\R^{2d-2}$ such that $\pi_0(K)$ has non-empty interior, where $\pi_0$ is the projection onto the first $d-1$ coordinates.  Then some finite union of rotated copies of 
    $$
    {\mathbb K}_0 = \bigcup_{({\bf v},{\bf a})\in K} L_{{\bf v},{\bf a}}
    $$
    is a Kakeya set in $\R^d$. 
\end{proposition}

\begin{proof}
    First, the closedness of ${\mathbb K}_0$ follows from a direct check using sequences. Since ${\mathbb K}_0$ is clearly bounded, it must be compact. 

Note that the projection $\pi_0(K)$ contains an open ball $B$ in $\R^{d-1}$. Thus, by assumption, ${\mathbb K}_0$ contains directions of the form $\{({\bf v}, 1): {\bf v}\in B\}$. Its image under the normalization map covers an open set on the $(d-1)$-dimensional unit sphere.  As the unit sphere is compact, we can rotate ${\mathbb K}_0$ finitely many times, taking the union, so that the resulting set of directions is the whole sphere. 
\end{proof}

For each $\lambda\in \R$, let $\pi_{\lambda}(x,y) = x+\lambda y$ denote the projection onto the 1-dimensional subspace with slope $\lambda$.   We will use the Bescovitch projection theorem (see \cite[Corollary 6.14]{falconer1985geometry}), which we recall as follows:

\begin{theorem}\label{BPT}
    Let $K$ be a compact set in $\R^2$ such that $0<{\mathcal H}^1(K)<\infty$. Suppose there are two distinct  $\lambda_1,\lambda_2 \in \R$ such that the orthogonal projection, $\pi_{\lambda_i}(K)$, has zero Lebesgue measure for each $i =1,2$. Then for almost all $\lambda\in\R$, the orthogonal projections have zero Lebesgue measure. 
\end{theorem}

The following proposition gives us the criteria for the compact set $K$ to produce our desired non-sticky Kakeya set of measure zero:
\begin{proposition}\label{prop-K}
Let $K\subset[0,1]^2$ be a compact set such that 
\begin{enumerate}
    \item $\pi_0(K)$ contains an interval;
    \item $ \pdim K>1$;
    \item ${\mathcal H}^1(K)<\infty$;
    \item There are two distinct $\lambda_1,\lambda_2$ such that the $\pi_{\lambda_i}(K)$ has Lebesgue measure zero. 

\end{enumerate}
Then some finite union of rotated copies of 
$$
{\mathbb K}_0 = \bigcup_{(v,a)\in K} L_{v,a}
$$
form a compact non-sticky Kakeya set of measure zero. 
\end{proposition}

\begin{proof}
    By Proposition \ref{prop0}, this set is a compact Kakeya set. Since ${\mathcal A}_{\mathbb{K}}$ contains $K$, by (2), $\pdim {\mathcal A}_{\mathbb {K}} > 1$, hence the Kakeya set is non-sticky. Since (1) and (3) imply that $0<{\mathcal H}^1(K)<\infty$, (4) implies we can apply the Besicovitch projection theorem \ref{BPT} to conclude that for almost all $t\in[0,1]$, the Lebesgue measure of the projection $m^1(\pi_{t} (K)) = 0.$
   As the set 
    $$
    {\mathbb K}_0\cap \{(x,y)\in \R^2: y=t\} = \{a + vt: (v,a)\in K\}\times  \{t\} =  \pi_t(K)\times \{t\},
    $$
    almost all slices have Lebesgue measure zero. By Fubini's theorem, $m^2({\mathbb K}_0) = 0$. Therefore, the Kakeya set also has measure zero. 
\end{proof}

\subsection{Packing dimension of fractal cubes} Because of Proposition \ref{prop-K}, our goal is to construct a Cantor set satisfying the specified conditions. We will build our set using fractal squares, and in this subsection, we obtain a lower bound for their packing dimension. Since we could not find a direct reference in the literature about this topic, we make a careful treatment here for the sake of completeness.

Let $(N_k)_{k = 1}^{\infty}$ and $(M_k)_{k=1}^{\infty}$ be two sequences of positive integers with $N_k\ge 2$ and $1<M_k<N_k^d$. The {\bf fractal cubes} associated with $(N_k)_{k = 1}^{\infty}$ and $(M_k)_{k=1}^{\infty}$ are the fractal constructed via the following inductive procedure.

Let $C_0 = [0,1]^d$. Divide $C_0$ into $N_1^d$ many cubes of side length $N_1^{-1}$, and choose $M_1$ many such cubes. Let $C_1$ be the collection of all these cubes.  Suppose $C_k$ has been chosen for some $k\ge 1$. For each cube $Q$ in $C_k$, divide $Q$ into $N_{k+1}^d$ many cubes of side length $(N_1\cdots N_{k}N_{k+1})^{-1}$, and choose $M_{k+1}$ many distinct cubes. Let $C_{k+1}$ be the collection of all the cubes chosen from all the $Q$ in $C_k$. In this construction, $C_k$ forms a decreasing sequence of compact sets.  The fractal cubes associated with $(N_k)_{k = 1}^{\infty}$ and $(M_k)_{k=1}^{\infty}$ are the set 
$$
C = \bigcap_{n=1}^{\infty} C_k.
$$

Next we provide a lower bound on the packing dimension for the fractal cubes. Here, we would like to thank Ying Xiong for helping us navigate the correct literature. 

\begin{proposition}\label{prop-fractal-cube}
    For the fractal cubes constructed above, 
    $$
    \pdim (C)\ge\limsup_{k\to\infty} \frac{\log (M_1\cdots M_k)}{\log (N_1\cdots N_k)}.
    $$
\end{proposition}

\begin{proof}
    Let $\mu$ be the probability measure that assigns mass to each $Q$ in each $C_k$ by $\mu (Q) = (M_1\cdots M_k)^{-1}$, and extend $\mu$ to a Borel probability measure by Carath\'eodory's extension theorem.  Then $\mu$ is supported on $C$. Define 
    $$
    \pdim(\mu) = \inf\{\pdim (E): \mu (E^{c}) = 0, E \ \mbox{Borel}\}.
    $$
    Using this definition, $\pdim(C)\ge \pdim(\mu)$. By \cite{FLR02} (see also \cite[Theorem 2.6.5]{BKS2024}), $\pdim(\mu)\ge \overline{\edim}(\mu)$, where $\overline{\edim}(\mu)$ is the upper entropy dimension of $\mu$, which is defined as follows:
    $$
    \overline{\edim}(\mu) = \limsup_{r\to 0} \frac{H_r(\mu)}{-\log r}, 
 \hspace{15pt} H_r(\mu) = -\int \log \mu (B_r(x))d\mu (x). 
    $$
  Consider the sequence of scales $r_k = (N_1\cdots N_k)^{-1}$. There exists a constant $c_d$, depending on the ambient space $\R^d$, such that balls $B_{r_k}(x)$ intersect at most $c_d$  many cubes of side length $r_k$. If $x\in C$, then since $B_{r_k}(x)$ intersects at most $c_d$ many $Q$ in $C_k$, 
  $$
  \mu (B_{r_k}(x)) \le c_d \cdot \mu(Q) = c_d  \cdot (M_1\cdots M_k)^{-1}.  
  $$
Decomposing the integration over all $Q$ in $C_k$
  $$
  H_{r_k}(\mu)\ge -\sum_{Q\in C_k} \mu (Q)\cdot \log [c_d  \cdot (M_1\cdots M_k)^{-1}] = -\log c_d+\log (M_1\cdots M_k),
  $$
since $\mu (Q) = (M_1\cdots M_k)^{-1}$ and there are exactly $M_1\cdots M_k$ many $Q$ in $C_k$.   Consequently, 
  $$
  \overline{\edim}(\mu) ~\ge ~ \limsup_{k\to \infty} \frac{H_{r_k}(\mu)}{-\log r_k} ~\ge ~\limsup_{k\to \infty}\frac{\log (M_1\cdots M_k)}{\log (N_1\cdots N_k)}.
  $$
\end{proof}

Finally, we recall the dimension inequalities for packing and Hausdorff dimensions of products that we will need. For $E\subset\R^d $ and $F\subset \R^{d'}$,
\begin{equation}\label{hdimeq1}
\hdim(E)+\hdim(F) \le \hdim(E\times F) \le \hdim(E)+\pdim(F), 
\end{equation}
\begin{equation}\label{hdimeq2}
 \hdim(E)+\pdim(F)  \le \pdim(E\times F)\le \pdim(E)+\pdim(F).  
\end{equation}
These inequalities can be found in \cite[p.62]{BKS2024}, and the proof of (\ref{hdimeq2}) can be found in \cite{Tricot82}.

\section{Cantor set construction for non-sticky Kakeya sets}\label{section3}

In this section, we describe a Cantor set construction that will lead to the existence of $K$ in Proposition \ref{prop-K}. Fix a sequence of positive integers $\{n_k\}_{k=1}^{\infty}$ with the property that 
$$
n_{2k}\le n_{2k-1} \hspace{10pt} \text{for all } k\ge 1.
$$

Let $C_0 = [0,1]^2$. We subdivide $C_0$ into $4^{2n_1} = 2^{4n_1}$ squares, each of side length $4^{-n_1}$, and choose $2^{3n_1}$ squares satisfying the following condition:

\smallskip

\begin{enumerate}
    \item [${\bf (*)}_1$] In each column  $[j4^{-n_1}, (j+1)4^{-n_1}]\times [0,1]$, for $j = 0,1,\cdots, 4^{n_1}-1$, exactly $2^{n_1}$ squares are chosen.
\end{enumerate}

After $C_1$ is chosen, we subdivide each square in $C_1$ into $2^{4n_2}$ subsquares, each of side length $4^{-(n_1+n_2)}$, and choose  $2^{n_2}$ subsquares satisfying the following alternative pair of conditions:
\smallskip

\begin{enumerate}
    \item [${\bf (* ~*)}_1$] In each column $[j4^{-(n_1+n_2)}, (j+1)4^{-(n_1+n_2)}]\times [0,1]$, for $j = 0,1,\cdots, 4^{n_1+n_2}-1$, at least one square is chosen.

    \item [${\bf (***)}_1$] For both the $45^{\circ}$ and $-45^{\circ}$ orthogonal projections, there exist at least two squares whose projections overlap exactly.
\end{enumerate}
\smallskip

\noindent Condition ${\bf (***)}_1$ can easily be achieved by choosing two pairs of squares that are aligned diagonally. For ${\bf (* ~*)}_1$, we have the following lemma:

\begin{lem}\label{lemma1}
If $n_1\ge n_2$, then we can choose squares satisfying ${\bf (*~*)}_1$.
\end{lem}

\begin{proof}
    Fix a column at the $(4^{-n_1})$-scale and denote it by $R_k = [k4^{-n_1}, (k+1)4^{-n_1}]\times [0,1]$. Using $({\bf *})_1$, there are exactly $2^{n_1}$ squares of side length $4^{-n_1}$, denoted by $Q_1,\cdots, Q_{2^{n_1}}$. By construction, we will choose $2^{n_2}$ subsquares from each square $Q_k$, where $k = 1,\cdots, 2^{n_1}$. Hence, there will be exactly $2^{n_1+n_2}$ subsquares of side length $4^{-(n_1+n_2)}$ in $R_k$.

     We now subdivide $R_k$ into $4^{n_2}$ subcolumns of width $4^{-(n_1+n_2)}$. Since $n_1\ge n_2$,   we have $2^{n_1+n_2}\ge 4^{n_2}$, so that the number of subsquares in $R_k$ is greater than the number of subcolumns. Thus, we are guaranteed to be able to choose at least one square in each column $[j4^{-(n_1+n_2)}, (j+1)4^{-(n_1+n_2)}]\times [0,1]$ in $R_k$. Since this property holds for all $k$, ${\bf (*~*)}_1$ holds for all columns of width $4^{-(n_1+n_2)}$.
\end{proof}

Finally, let $C_2$ be the collection of all chosen squares of side length $4^{-(n_1+n_2)}$. See Figure \ref{figure1} for an illustration with $n_1 = n_2 = 1$.

\begin{figure}[h]
\centering

\tikzset{every picture/.style={line width=0.75pt}} 

\begin{tikzpicture}[x=0.75pt,y=0.75pt,yscale=-1,xscale=1]

\draw    (10,10) -- (10,170.5) ;
\draw    (170,10) -- (170,170.5) ;
\draw    (10,10) -- (170,10) ;
\draw    (10,170) -- (170,170) ;
\draw    (50,10) -- (50,170.5) ;
\draw    (90,10) -- (90,170.5) ;
\draw    (130,10) -- (130,170.5) ;
\draw    (10,50) -- (170,50) ;
\draw    (10,90) -- (170,90) ;
\draw    (10,130) -- (170,130) ;
\draw  [fill={rgb, 255:red, 0; green, 0; blue, 0 }  ,fill opacity=1 ] (10,10) -- (50,10) -- (50,50) -- (10,50) -- cycle ;
\draw    (210,10) -- (210,170.5) ;
\draw    (370,10) -- (370,170.5) ;
\draw    (210,10) -- (370,10) ;
\draw    (210,170) -- (370,170) ;
\draw    (250,10) -- (250,170.5) ;
\draw    (290,10) -- (290,170.5) ;
\draw    (330,10) -- (330,170.5) ;
\draw    (210,50) -- (370,50) ;
\draw    (210,90) -- (370,90) ;
\draw    (210,130) -- (370,130) ;
\draw  [fill={rgb, 255:red, 209; green, 209; blue, 209 }  ,fill opacity=1 ] (210,10) -- (250,10) -- (250,50) -- (210,50) -- cycle ;
\draw  [fill={rgb, 255:red, 0; green, 0; blue, 0 }  ,fill opacity=1 ] (90,10) -- (130,10) -- (130,50) -- (90,50) -- cycle ;
\draw  [fill={rgb, 255:red, 0; green, 0; blue, 0 }  ,fill opacity=1 ] (130,50) -- (170,50) -- (170,90) -- (130,90) -- cycle ;
\draw  [fill={rgb, 255:red, 0; green, 0; blue, 0 }  ,fill opacity=1 ] (50,50) -- (90,50) -- (90,90) -- (50,90) -- cycle ;
\draw  [fill={rgb, 255:red, 0; green, 0; blue, 0 }  ,fill opacity=1 ] (10,90) -- (50,90) -- (50,130) -- (10,130) -- cycle ;
\draw  [fill={rgb, 255:red, 0; green, 0; blue, 0 }  ,fill opacity=1 ] (130,90) -- (170,90) -- (170,130) -- (130,130) -- cycle ;
\draw  [fill={rgb, 255:red, 0; green, 0; blue, 0 }  ,fill opacity=1 ] (90,130) -- (130,130) -- (130,170) -- (90,170) -- cycle ;
\draw  [fill={rgb, 255:red, 0; green, 0; blue, 0 }  ,fill opacity=1 ] (50,130) -- (90,130) -- (90,170) -- (50,170) -- cycle ;
\draw  [fill={rgb, 255:red, 209; green, 209; blue, 209 }  ,fill opacity=1 ] (290,10) -- (330,10) -- (330,50) -- (290,50) -- cycle ;
\draw  [fill={rgb, 255:red, 209; green, 209; blue, 209 }  ,fill opacity=1 ] (250,50) -- (290,50) -- (290,90) -- (250,90) -- cycle ;
\draw  [fill={rgb, 255:red, 209; green, 209; blue, 209 }  ,fill opacity=1 ] (330,50) -- (370,50) -- (370,90) -- (330,90) -- cycle ;
\draw  [fill={rgb, 255:red, 209; green, 209; blue, 209 }  ,fill opacity=1 ] (330,90) -- (370,90) -- (370,130) -- (330,130) -- cycle ;
\draw  [fill={rgb, 255:red, 209; green, 209; blue, 209 }  ,fill opacity=1 ] (210,90) -- (250,90) -- (250,130) -- (210,130) -- cycle ;
\draw  [fill={rgb, 255:red, 209; green, 209; blue, 209 }  ,fill opacity=1 ] (250,130) -- (290,130) -- (290,170) -- (250,170) -- cycle ;
\draw  [fill={rgb, 255:red, 209; green, 209; blue, 209 }  ,fill opacity=1 ] (290,130) -- (330,130) -- (330,170) -- (290,170) -- cycle ;
\draw    (220,10) -- (220,170.5) ;
\draw    (230,10) -- (230,170.5) ;
\draw    (240,10) -- (240,170.5) ;
\draw    (260,10) -- (260,170.5) ;
\draw    (270,10) -- (270,170.5) ;
\draw    (280,10) -- (280,170.5) ;
\draw    (320,10) -- (320,170.5) ;
\draw    (300,10) -- (300,170.5) ;
\draw    (310,10) -- (310,170.5) ;
\draw    (340,10) -- (340,170.5) ;
\draw    (360,10) -- (360,170.5) ;
\draw    (350,10) -- (350,170.5) ;
\draw    (210,20) -- (370,20) ;
\draw    (210,30) -- (370,30) ;
\draw    (210,40) -- (370,40) ;
\draw    (210,60) -- (370,60) ;
\draw    (210,70) -- (370,70) ;
\draw    (210,80) -- (370,80) ;
\draw    (210,100) -- (370,100) ;
\draw    (210,110) -- (370,110) ;
\draw    (210,120) -- (370,120) ;
\draw    (210,140) -- (370,140) ;
\draw    (210,150) -- (370,150) ;
\draw    (210,160) -- (370,160) ;
\draw  [fill={rgb, 255:red, 0; green, 0; blue, 0 }  ,fill opacity=1 ] (210.5,21) -- (219.5,21) -- (219.5,29.5) -- (210.5,29.5) -- cycle ;
\draw  [fill={rgb, 255:red, 0; green, 0; blue, 0 }  ,fill opacity=1 ] (230.5,31) -- (239.5,31) -- (239.5,39.5) -- (230.5,39.5) -- cycle ;
\draw  [fill={rgb, 255:red, 0; green, 0; blue, 0 }  ,fill opacity=1 ] (240.5,101) -- (249.5,101) -- (249.5,109.5) -- (240.5,109.5) -- cycle ;
\draw  [fill={rgb, 255:red, 0; green, 0; blue, 0 }  ,fill opacity=1 ] (220.5,111) -- (229.5,111) -- (229.5,119.5) -- (220.5,119.5) -- cycle ;
\draw  [fill={rgb, 255:red, 0; green, 0; blue, 0 }  ,fill opacity=1 ] (280.5,141) -- (289.5,141) -- (289.5,149.5) -- (280.5,149.5) -- cycle ;
\draw  [fill={rgb, 255:red, 0; green, 0; blue, 0 }  ,fill opacity=1 ] (250.5,131) -- (259.5,131) -- (259.5,139.5) -- (250.5,139.5) -- cycle ;
\draw  [fill={rgb, 255:red, 0; green, 0; blue, 0 }  ,fill opacity=1 ] (290.5,151) -- (299.5,151) -- (299.5,159.5) -- (290.5,159.5) -- cycle ;
\draw  [fill={rgb, 255:red, 0; green, 0; blue, 0 }  ,fill opacity=1 ] (310.5,161) -- (319.5,161) -- (319.5,169.5) -- (310.5,169.5) -- cycle ;
\draw  [fill={rgb, 255:red, 0; green, 0; blue, 0 }  ,fill opacity=1 ] (330.5,61) -- (339.5,61) -- (339.5,69.5) -- (330.5,69.5) -- cycle ;
\draw  [fill={rgb, 255:red, 0; green, 0; blue, 0 }  ,fill opacity=1 ] (360.5,81) -- (369.5,81) -- (369.5,89.5) -- (360.5,89.5) -- cycle ;
\draw  [fill={rgb, 255:red, 0; green, 0; blue, 0 }  ,fill opacity=1 ] (350.5,101) -- (359.5,101) -- (359.5,109.5) -- (350.5,109.5) -- cycle ;
\draw  [fill={rgb, 255:red, 0; green, 0; blue, 0 }  ,fill opacity=1 ] (340.5,121) -- (349.5,121) -- (349.5,129.5) -- (340.5,129.5) -- cycle ;
\draw  [fill={rgb, 255:red, 0; green, 0; blue, 0 }  ,fill opacity=1 ] (270.5,61) -- (279.5,61) -- (279.5,69.5) -- (270.5,69.5) -- cycle ;
\draw  [fill={rgb, 255:red, 0; green, 0; blue, 0 }  ,fill opacity=1 ] (260.5,71) -- (269.5,71) -- (269.5,79.5) -- (260.5,79.5) -- cycle ;
\draw  [fill={rgb, 255:red, 0; green, 0; blue, 0 }  ,fill opacity=1 ] (300.5,21) -- (309.5,21) -- (309.5,29.5) -- (300.5,29.5) -- cycle ;
\draw  [fill={rgb, 255:red, 0; green, 0; blue, 0 }  ,fill opacity=1 ] (320.5,41) -- (329.5,41) -- (329.5,49.5) -- (320.5,49.5) -- cycle ;
\end{tikzpicture}
\caption{Examples of first (left) and second (right) iterations of the construction with $n_1=n_2 = 1$ and with ${\bf (***)}_1$ also satisfied.}\label{figure1}
\end{figure}
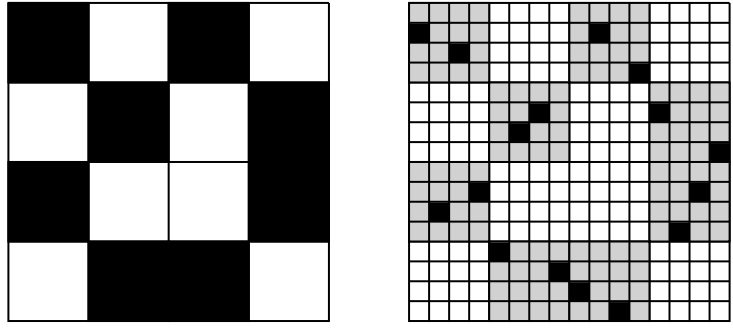

    \smallskip

\noindent{\bf Inductive step.} In the inductive step, suppose that $C_{2k-1}$ and $C_{2k}$ have been chosen.  For each square $Q_k$ in $C_{2k}$, we subdivide $Q_k$ into $2^{4n_{2k+1}}$ many squares and choose $2^{3n_{k+1}}$ of them. We will require 

\smallskip

\begin{enumerate}
    \item [${\bf (*)}_k$] For all $Q_k$ in $C_{2k}$, and for each column of width $4^{-(n_1+\cdots+n_{2k+1})}$ of the subdivision of $Q_k$,  exactly $2^{n_{k+1}}$ squares from each column are chosen.
\end{enumerate}
\smallskip

\noindent Denote the collection of all of these squares obtained from all $Q_k$ by $C_{2k+1}$.

We now further subdivide each square in $C_{2k+1}$ into $2^{4n_{2k+2}}$ many subsquares and choose $2^{n_{2k+2}}$ many squares satisfying the following conditions:
\smallskip

\begin{enumerate}
    \item [${\bf (*~*)}_k$] In each column of width $4^{-(n_1+\cdots+n_{2k+2})}$ in each $Q_k$ from $C_{2k}$, at least one square is chosen.
    \item [${\bf (***)}_k$] For both the $45^{\circ}$ and $-45^{\circ}$ orthogonal projections, there exist at least two squares whose projections overlap exactly.
\end{enumerate}
\smallskip

By a similar argument as Lemma \ref{lemma1}, provided that $n_{2k+1}\ge n_{2k+2}$, ${\bf (*~*)}_k$ can always be satisfied. We let $C_{2k+2}$ be the collection of chosen squares at scale $4^{-(n_1+\cdots+n_{2k+2})}$.   This completes the inductive step. Our Cantor set is
$$
C = \bigcap_{k=1}^{\infty} C_{k},
$$
and we will say that $C$ is a Cantor set generated by the sequence $\{n_k\}_{k=1}^{\infty}$. We notice that the construction of squares in $C_{2k+2}$ may be completely independent in each $Q_k$ from $C_{2k}$, so we will also consider the following additional condition:

\smallskip

$({\mathsf M}):$ The patterns of squares at scale $4^{-(n_1+\cdots +n_{2k+2})}$ in $Q_k$ are the same for all  $Q_k$ in $C_{2k}$.

\smallskip

\noindent This will make $C$ a {\it Moran fractal cube}, which is a commonly used name in the literature.

First, $C$ satisfies condition (1) of Proposition \ref{prop-K}:

\begin{proposition}\label{prop_proj}
    $\pi_0(C) = [0,1]$.  
\end{proposition}

\begin{proof}
    For all $x_0\in[0,1]$, the line $x=x_0$ always passes through some square $Q_2$ from $C_2$, by $({\bf *})_1$ and $({\bf *~*})_1$. By induction with $({\bf *})_k$ and $({\bf *~*})_{k}$, we can find a nested sequence of cubes $Q_2\supset Q_4\supset\cdots$, where $Q_{2k}$ is drawn from $C_{2k}$, such that the line $x=x_0$ intersects all $Q_{2k}$. Hence, the line $x=x_0$ intersects $C$.
\end{proof}

\subsection{Hausdorff and packing dimension.} The following is our main result about the dimensions of the Cantor set $C$, which when combined with Proposition \ref{prop_proj}, gives a compact set satisfying all the conditions of Proposition \ref{prop-K}.

\begin{proposition}\label{prop3.3}
    There exists a sequence $\{n_k\}_{k=1}^{\infty}$ with $n_{2k-1}\ge n_{2k}$ for all $k$ such that the Cantor set $C$ it generates satisfies
    \begin{enumerate}
    \item $0<{\mathcal H}^1(C)<\infty$, i.e., $\hdim C = 1$;
    \item $\pdim(C) >1$; 
    \item $m^{1} (\pi_{\pm} (C)) = 0$,
    \end{enumerate}
    where $\pi_{\pm}$ are, respectively, the orthogonal projections onto $y=\pm x$. 
\end{proposition}

\begin{proof} It suffices to choose $n_{2k-1} = n_{2k}$, for all $k\ge 1$. We remark that a more general proof is possible, but we do not pursue it here.  Proceeding inductively  on pairs of indices, we first pick $n_1=n_2 = 1$, and we repeat $m_1$ times the construction with the pair $(n_1,n_2)$, so that 
\begin{equation}\label{eq_m-1-condition}
\left(1-\frac{1}{4^{n_1+n_2}}\right)^{m_1}<\frac12.
\end{equation}
Assuming that the construction has been repeated $m_k$ times on the pair $(n_{2k-1}, n_{2k})$, so that
\begin{equation}\label{eq_m-k-condition}
\left(1-\frac{1}{4^{n_{2k-1}+n_{2k}}}\right)^{m_k}<\frac{1}{2^{k}},
\end{equation}
we choose $(n_{2k+1}, n_{2k+2})$ satisfying 
\begin{equation}\label{eq_n-k-condition}
\frac{\sum_{j=1}^{k}m_j(3n_{2j-1}+n_{2j})}{n_{2k+1}}< \frac{1}{100} \ \hspace{10pt} \mbox{and} \hspace{10pt}  \frac{\sum_{j=1}^{k}m_j(2n_{2j-1}+2n_{2j})}{n_{2k+1}}< \frac{1}{100}
\end{equation}
and repeat $m_{k+1}$ times the construction on the pair $(n_{2k+1},n_{2k+2})$, so that 
\begin{equation}\label{eq_m-{k+1}-condition}
\left(1-\frac{1}{4^{n_{2k+1}+n_{2k+2}}}\right)^{m_{k+1}}<\frac{1}{2^{k+1}}.
\end{equation}
The resulting Cantor set $C$ is generated by the sequence
$$
(\underbrace{(n_1,n_2),\cdots, (n_1,n_2)}_{m_1 \ \text{times}}, \cdots, \underbrace{(n_{2k-1},n_{2k}),\cdots, (n_{2k-1},n_{2k})}_{m_k \ \text{times}}, \cdots),
$$
so that $(\ast)_k$, $(\ast~\ast)_k$, and $(\ast\ast\ast)_k$ are satisfied, for all $k\ge 1$.

\medskip

We next verify each of the conditions:

\noindent (1).    Let $\delta_k = 4^{-[m_1(n_1+n_2)\cdots+ m_k(n_{2k-1}+n_{2k})]}$ and $\delta_k' = \delta_k\cdot 4^{-n_{2k+1}}$. Denote by  $M_k$ and $M_k'$ the total number of squares of side lengths $\delta_k$ and $\delta_k'$, respectively. By our construction,  $M_{k} = 2^{m_1(3n_1+n_2)+m_2(3n_3+n_4)+\cdots+m_k(3n_{2k-1}+n_{2k})}$. Note that the squares of side length $\delta_k$ have diameter $\sqrt{2} \delta_{k}$ and form a covering of the set $C$, so 
 \begin{equation}\label{eq_hdim}
    \begin{aligned}
    {\mathcal H}^1(C)\le & \sqrt{2} \cdot  \liminf_{k\to\infty} (M_{k}\cdot \delta_{k})\\
      = & \sqrt{2} \liminf_{k\to\infty} 2^{(m_1(n_1-n_2)+\cdots+m_k(n_{2k-1}-n_{2k}))}  = \sqrt{2}<\infty,
   \end{aligned}
    \end{equation}
since we have chosen $n_{2k-1} = n_{2k}$, for all $k\ge 1$. On the other hand, ${\mathcal H}^1(C) \geq {\mathcal H}^1(\pi_0(C))$, so by Proposition \ref{prop_proj}, ${\mathcal H}^1(C)>0$.

\smallskip

\noindent(2). Using Proposition \ref{prop-fractal-cube} for $C$, which is a special case of fractal cubes, the packing dimension of $C$ is 
$$
\begin{aligned}
\pdim(C)\ge& \limsup_{k\to\infty} \frac{\log M_{k}' }{-\log \delta_{k}'}\\
=& \limsup_{k\to\infty} \frac{\log 2^{[m_1(3n_1+n_2)+\cdots+m_k(n_{2k-1}+ n_{2k})]+3n_{2k+1}}}{\log 2^{2[m_1(n_1+n_2)\cdots +m_k(n_{2k-1}+n_{2k})+n_{2k+1}]}}\\
= & \limsup_{k\to\infty} \frac{\frac{\sum_{j=1}^{k}m_j(3n_{2j-1}+n_{2j})}{n_{2k+1}} +3}{\frac{\sum_{j=1}^{k}m_j(2n_{2j-1}+2n_{2j})}{n_{2k+1}}+2}>1,
\end{aligned}
$$
where we used (\ref{eq_n-k-condition}) in the last line.

\smallskip

\noindent (3). To establish (3), we will require $C$ to satisfy $({\mathsf M})$.  We notice that the Cantor set $C$ is contained in $2^{3n_1+n_2}$ squares of side length $4^{-(n_1+n_2)}$. By ${\bf (***)}_1$,  there is at least one pair of squares that overlap exactly when $C$ is projected to $45^{\circ}$, so $\pi_{+}(C)$ is contained in at most $4^{n_1+n_2}-1$ many intervals of length $\sqrt{2}\cdot4^{-(n_1+n_2)}$. Hence, 
$$
m^1(\pi_+(C)) \le \sqrt{2} \cdot \left(1-\frac1{4^{n_1+n_2}}\right).
$$
Note that for each square of side length $4^{-(n_1+n_2)}$ in the first iteration of the construction, there are $2^{3n_1+n_2}$ squares of side length $4^{-2(n_1+n_2)}$ in the second stage (using indices $(n_1,n_2)$ still). 

As $C$ satisfies $({\mathsf M})$, we can write  is as
$$
C = \bigcup_{i=1}^{2^{3n_1+n_2}} \frac{1}{4^{n_1+n_2}} (K_1+d_{1,i})
$$
where $K_1$ is the Cantor set generated by sequences $\{n_k\}_{k=3}^{\infty}$. Inductively, 
$$
K_{j+1} = \bigcup_{i=1}^{2^{3n_{2j+1}+n_{2j+2}}} \frac{1}{4^{n_{2j+1}+n_{2j+2}}} (K_j+d_{1,i})
$$
Let $\pi_{\pm}$ be the orthogonal projection onto $y=\pm x$ respectively. We will consider only  $\pi_+$ as $\pi_-$ follows from the same argument. 
$$
\pi_{+}C = \bigcup_{i=1}^{2^{3n_1+n_2}} \frac{1}{4^{n_1+n_2}} (\pi_{+}(K_1)+\pi_+(d_{1,i}))
$$
From $(\ast\ast\ast)_1$, there exists at least $i\ne i'$ such that  $\pi_{+}d_{1,i} = \pi_{+}d_{1,i'}$. Hence,   
$$
|\pi_{+}(C)| \le |\pi_+ (K_1)| \cdot  \left(1-\frac{1}{4^{n_1+n_2}}\right).
$$
Here, $|\cdot|$ denotes the one-dimensional Lebesgue measure of the set. Since $(\ast\ast\ast)_k$ holds for all $k$, we have 
$$
|\pi_{+}C| \le \prod_{j=1}^{\infty}\left(1-\frac{1}{4^{n_{2j-1}+n_{2j}}}\right)^{m_j} \le \prod_{j=1}^{\infty} \frac{1}{2^j} = 0
$$
by (\ref{eq_m-k-condition}). This completes the proof. 
\end{proof}


\begin{proof}[Proof of Theorem \ref{main theorem}] On $\R^2$, we just take the Cantor set $C$ from Proposition \ref{prop3.3}. As remarked, Proposition \ref{prop_proj} gives us condition (1) of Proposition \ref{prop-K}, and Proposition \ref{prop3.3} gives us the conditions (2), (3), and (4). Hence, we can use $C$ to construct a compact non-sticky Kakeya set of measure zero. 

To prove the statement on $\R^d$, let $C$ be the same Cantor set, and define a compact set on $\R^{2d-2}$ by
$$
C' = \{(v,\mathbf{x},a,{\bf 0}): (v,a)\in C, \mathbf{x}\in [0,1]^{d-2}\},
$$
where ${\bf 0}$ is the $(d-2)$-dimensional zero vector. Observe that $C'$ is isometric to  $ C\times [0,1]^{d-2}$. Using the product inequality for Hausdorff dimension (\ref{hdimeq1}) with $E = C$ and $F = [0,1]^{d-2}$ and Proposition \ref{prop3.3} (1), we see that 
$$
\hdim(C') = d-1.
$$
Using the product inequality for packing dimension (\ref{hdimeq2}) with $E = [0,1]^{d-2}$ and $F = C$ and Proposition \ref{prop3.3} (2), we see that 
$$
\pdim(C')>d-1.
$$
Furthermore, the projection of $C'$ onto the first $d-1$ coordinates is exactly $[0,1]^{d-1}$,
so by Proposition \ref{prop0}, a finite union of rotated copies of
$$
{\mathbb K}_0 = \bigcup_{(v,\mathbf{x},a,{\bf 0})\in C'}\{(a,{\bf 0},0)+t(v,\mathbf{x},1): t\in [0,1]\},
$$
is a Kakeya set in $\R^d$. Moreover,
the set ${\mathcal A}_{{\mathbb K}_0} = C'$, so this Kakeya set is non-sticky.

Finally, for each $t\in[0,1]$, the slice ${\mathbb K}_0\cap \{z=t\} = \pi_t(C)\times [0,1]^{d-2}$. By Proposition \ref{prop3.3} (3) and the Besicovitch projection theorem \ref{BPT}, for almost all $t$, $m^1(\pi_t(C)) = 0$, so almost all slices have measure zero. Thus, by Fubini's theorem, ${\mathbb K}_0$ has measure zero, so the Kakeya set does, as well.
\end{proof}

\begin{figure}[h]
\centering
\tikzset{every picture/.style={line width=0.75pt}} 

\tikzset{every picture/.style={line width=0.75pt}} 

\begin{tikzpicture}[x=0.75pt,y=0.75pt,yscale=-0.9,xscale=0.9]

\draw    (30,210) -- (30,370.5) ;
\draw    (190,210) -- (190,370.5) ;
\draw    (30,210) -- (190,210) ;
\draw    (30,370) -- (190,370) ;
\draw    (70,210) -- (70,370.5) ;
\draw    (110,210) -- (110,370.5) ;
\draw    (150,210) -- (150,370.5) ;
\draw    (30,250) -- (190,250) ;
\draw    (30,290) -- (190,290) ;
\draw    (30,330) -- (190,330) ;
\draw  [fill={rgb, 255:red, 209; green, 209; blue, 209 }  ,fill opacity=1 ] (150,330) -- (190,330) -- (190,370) -- (150,370) -- cycle ;
\draw  [fill={rgb, 255:red, 209; green, 209; blue, 209 }  ,fill opacity=1 ] (110,210) -- (150,210) -- (150,250) -- (110,250) -- cycle ;
\draw  [fill={rgb, 255:red, 209; green, 209; blue, 209 }  ,fill opacity=1 ] (30,250) -- (70,250) -- (70,290) -- (30,290) -- cycle ;
\draw  [fill={rgb, 255:red, 209; green, 209; blue, 209 }  ,fill opacity=1 ] (150,250) -- (190,250) -- (190,290) -- (150,290) -- cycle ;
\draw  [fill={rgb, 255:red, 209; green, 209; blue, 209 }  ,fill opacity=1 ] (70,210) -- (110,210) -- (110,250) -- (70,250) -- cycle ;
\draw  [fill={rgb, 255:red, 209; green, 209; blue, 209 }  ,fill opacity=1 ] (30,290) -- (70,290) -- (70,330) -- (30,330) -- cycle ;
\draw  [fill={rgb, 255:red, 209; green, 209; blue, 209 }  ,fill opacity=1 ] (70,330) -- (110,330) -- (110,370) -- (70,370) -- cycle ;
\draw  [fill={rgb, 255:red, 209; green, 209; blue, 209 }  ,fill opacity=1 ] (110,290) -- (150,290) -- (150,330) -- (110,330) -- cycle ;
\draw    (40,210) -- (40,370.5) ;
\draw    (50,210) -- (50,370.5) ;
\draw    (60,210) -- (60,370.5) ;
\draw    (80,210) -- (80,370.5) ;
\draw    (90,210) -- (90,370.5) ;
\draw    (100,210) -- (100,370.5) ;
\draw    (140,210) -- (140,370.5) ;
\draw    (120,210) -- (120,370.5) ;
\draw    (130,210) -- (130,370.5) ;
\draw    (160,210) -- (160,370.5) ;
\draw    (180,210) -- (180,370.5) ;
\draw    (170,210) -- (170,370.5) ;
\draw    (30,220) -- (190,220) ;
\draw    (30,230) -- (190,230) ;
\draw    (30,240) -- (190,240) ;
\draw    (30,260) -- (190,260) ;
\draw    (30,270) -- (190,270) ;
\draw    (30,280) -- (190,280) ;
\draw    (30,300) -- (190,300) ;
\draw    (30,310) -- (190,310) ;
\draw    (30,320) -- (190,320) ;
\draw    (29.5,341) -- (189.5,341) ;
\draw    (30,350) -- (190,350) ;
\draw    (30,360) -- (190,360) ;
\draw  [fill={rgb, 255:red, 0; green, 0; blue, 0 }  ,fill opacity=1 ] (30.5,261) -- (39.5,261) -- (39.5,269.5) -- (30.5,269.5) -- cycle ;
\draw  [fill={rgb, 255:red, 0; green, 0; blue, 0 }  ,fill opacity=1 ] (40.5,281) -- (49.5,281) -- (49.5,289.5) -- (40.5,289.5) -- cycle ;
\draw  [fill={rgb, 255:red, 0; green, 0; blue, 0 }  ,fill opacity=1 ] (60.5,321) -- (69.5,321) -- (69.5,329.5) -- (60.5,329.5) -- cycle ;
\draw  [fill={rgb, 255:red, 0; green, 0; blue, 0 }  ,fill opacity=1 ] (50.5,291) -- (59.5,291) -- (59.5,299.5) -- (50.5,299.5) -- cycle ;
\draw  [fill={rgb, 255:red, 0; green, 0; blue, 0 }  ,fill opacity=1 ] (80.5,351) -- (89.5,351) -- (89.5,359.5) -- (80.5,359.5) -- cycle ;
\draw  [fill={rgb, 255:red, 0; green, 0; blue, 0 }  ,fill opacity=1 ] (90.5,331) -- (99.5,331) -- (99.5,339.5) -- (90.5,339.5) -- cycle ;
\draw  [fill={rgb, 255:red, 0; green, 0; blue, 0 }  ,fill opacity=1 ] (160.5,341) -- (169.5,341) -- (169.5,349.5) -- (160.5,349.5) -- cycle ;
\draw  [fill={rgb, 255:red, 0; green, 0; blue, 0 }  ,fill opacity=1 ] (170.5,361) -- (179.5,361) -- (179.5,369.5) -- (170.5,369.5) -- cycle ;
\draw  [fill={rgb, 255:red, 0; green, 0; blue, 0 }  ,fill opacity=1 ] (150.5,251) -- (159.5,251) -- (159.5,259.5) -- (150.5,259.5) -- cycle ;
\draw  [fill={rgb, 255:red, 0; green, 0; blue, 0 }  ,fill opacity=1 ] (170.5,271) -- (179.5,271) -- (179.5,279.5) -- (170.5,279.5) -- cycle ;
\draw  [fill={rgb, 255:red, 0; green, 0; blue, 0 }  ,fill opacity=1 ] (90.5,221) -- (99.5,221) -- (99.5,229.5) -- (90.5,229.5) -- cycle ;
\draw  [fill={rgb, 255:red, 0; green, 0; blue, 0 }  ,fill opacity=1 ] (70.5,231) -- (79.5,231) -- (79.5,239.5) -- (70.5,239.5) -- cycle ;
\draw  [fill={rgb, 255:red, 0; green, 0; blue, 0 }  ,fill opacity=1 ] (130.5,301) -- (139.5,301) -- (139.5,309.5) -- (130.5,309.5) -- cycle ;
\draw  [fill={rgb, 255:red, 0; green, 0; blue, 0 }  ,fill opacity=1 ] (120.5,311) -- (129.5,311) -- (129.5,319.5) -- (120.5,319.5) -- cycle ;
\draw  [fill={rgb, 255:red, 0; green, 0; blue, 0 }  ,fill opacity=1 ] (110.5,211) -- (119.5,211) -- (119.5,219.5) -- (110.5,219.5) -- cycle ;
\draw  [fill={rgb, 255:red, 0; green, 0; blue, 0 }  ,fill opacity=1 ] (130.5,241) -- (139.5,241) -- (139.5,249.5) -- (130.5,249.5) -- cycle ;
\draw    (20,370) -- (20,341.5) ;
\draw [shift={(20,339.5)}, rotate = 90] [color={rgb, 255:red, 0; green, 0; blue, 0 }  ][line width=0.75]    (10.93,-3.29) .. controls (6.95,-1.4) and (3.31,-0.3) .. (0,0) .. controls (3.31,0.3) and (6.95,1.4) .. (10.93,3.29)   ;
\draw    (30,380) -- (57.5,380) ;
\draw [shift={(59.5,380)}, rotate = 180] [color={rgb, 255:red, 0; green, 0; blue, 0 }  ][line width=0.75]    (10.93,-3.29) .. controls (6.95,-1.4) and (3.31,-0.3) .. (0,0) .. controls (3.31,0.3) and (6.95,1.4) .. (10.93,3.29)   ;
\draw    (174,369.25) -- (333.5,365) ;
\draw    (85,370.25) -- (244,355) ;
\draw    (165,369.25) -- (323,342.5) ;
\draw    (94.75,370.25) -- (251.5,337) ;
\draw    (63.75,370.25) -- (217.5,325.5) ;
\draw    (125.75,370.25) -- (275.5,314.5) ;
\draw    (137.75,370.25) -- (285,305.5) ;
\draw    (52,370.25) -- (192,293) ;
\draw    (43.75,370.25) -- (179.25,285.5) ;
\draw    (178,370.25) -- (305.5,275) ;
\draw    (32.75,370.25) -- (151.5,263.5) ;
\draw    (152.75,370.25) -- (263.5,255.5) ;
\draw    (135.75,370.25) -- (235.5,244.5) ;
\draw    (77.75,370.25) -- (164,236) ;
\draw    (97,370.25) -- (168,227.5) ;
\draw    (117.75,370.25) -- (163,217) ;
\draw    (340,210) -- (340,370.5) ;
\draw    (500,210) -- (500,370.5) ;
\draw    (340,210) -- (500,210) ;
\draw    (340,370) -- (500,370) ;
\draw    (380,210) -- (380,370.5) ;
\draw    (420,210) -- (420,370.5) ;
\draw    (460,210) -- (460,370.5) ;
\draw    (340,250) -- (500,250) ;
\draw    (340,290) -- (500,290) ;
\draw    (340,330) -- (500,330) ;
\draw  [fill={rgb, 255:red, 209; green, 209; blue, 209 }  ,fill opacity=1 ] (340,250) -- (380,250) -- (380,290) -- (340,290) -- cycle ;
\draw  [fill={rgb, 255:red, 209; green, 209; blue, 209 }  ,fill opacity=1 ] (420,210) -- (460,210) -- (460,250) -- (420,250) -- cycle ;
\draw  [fill={rgb, 255:red, 209; green, 209; blue, 209 }  ,fill opacity=1 ] (460,290) -- (500,290) -- (500,330) -- (460,330) -- cycle ;
\draw  [fill={rgb, 255:red, 209; green, 209; blue, 209 }  ,fill opacity=1 ] (380,330) -- (420,330) -- (420,370) -- (380,370) -- cycle ;
\draw    (350,210) -- (350,370.5) ;
\draw    (360,210) -- (360,370.5) ;
\draw    (370,210) -- (370,370.5) ;
\draw    (390,210) -- (390,370.5) ;
\draw    (400,210) -- (400,370.5) ;
\draw    (410,210) -- (410,370.5) ;
\draw    (450,210) -- (450,370.5) ;
\draw    (430,210) -- (430,370.5) ;
\draw    (440,210) -- (440,370.5) ;
\draw    (470,210) -- (470,370.5) ;
\draw    (490,210) -- (490,370.5) ;
\draw    (480,210) -- (480,370.5) ;
\draw    (340,220) -- (500,220) ;
\draw    (340,230) -- (500,230) ;
\draw    (340,240) -- (500,240) ;
\draw    (340,260) -- (500,260) ;
\draw    (340,270) -- (500,270) ;
\draw    (340,280) -- (500,280) ;
\draw    (340,300) -- (500,300) ;
\draw    (340,310) -- (500,310) ;
\draw    (340,320) -- (500,320) ;
\draw    (339.5,341) -- (499.5,341) ;
\draw    (340,350) -- (500,350) ;
\draw    (340,360) -- (500,360) ;
\draw  [fill={rgb, 255:red, 0; green, 0; blue, 0 }  ,fill opacity=1 ] (340.5,261) -- (349.5,261) -- (349.5,269.5) -- (340.5,269.5) -- cycle ;
\draw  [fill={rgb, 255:red, 0; green, 0; blue, 0 }  ,fill opacity=1 ] (350.5,281) -- (359.5,281) -- (359.5,289.5) -- (350.5,289.5) -- cycle ;
\draw  [fill={rgb, 255:red, 0; green, 0; blue, 0 }  ,fill opacity=1 ] (370.5,271) -- (379.5,271) -- (379.5,279.5) -- (370.5,279.5) -- cycle ;
\draw  [fill={rgb, 255:red, 0; green, 0; blue, 0 }  ,fill opacity=1 ] (360.5,251) -- (369.5,251) -- (369.5,259.5) -- (360.5,259.5) -- cycle ;
\draw  [fill={rgb, 255:red, 0; green, 0; blue, 0 }  ,fill opacity=1 ] (390.5,361) -- (399.5,361) -- (399.5,369.5) -- (390.5,369.5) -- cycle ;
\draw  [fill={rgb, 255:red, 0; green, 0; blue, 0 }  ,fill opacity=1 ] (380.5,341) -- (389.5,341) -- (389.5,349.5) -- (380.5,349.5) -- cycle ;
\draw  [fill={rgb, 255:red, 0; green, 0; blue, 0 }  ,fill opacity=1 ] (470.5,321) -- (479.5,321) -- (479.5,329.5) -- (470.5,329.5) -- cycle ;
\draw  [fill={rgb, 255:red, 0; green, 0; blue, 0 }  ,fill opacity=1 ] (490.5,311) -- (499.5,311) -- (499.5,319.5) -- (490.5,319.5) -- cycle ;
\draw  [fill={rgb, 255:red, 0; green, 0; blue, 0 }  ,fill opacity=1 ] (430.5,241) -- (439.5,241) -- (439.5,249.5) -- (430.5,249.5) -- cycle ;
\draw  [fill={rgb, 255:red, 0; green, 0; blue, 0 }  ,fill opacity=1 ] (450.5,231) -- (459.5,231) -- (459.5,239.5) -- (450.5,239.5) -- cycle ;
\draw  [fill={rgb, 255:red, 0; green, 0; blue, 0 }  ,fill opacity=1 ] (480.5,291) -- (489.5,291) -- (489.5,299.5) -- (480.5,299.5) -- cycle ;
\draw  [fill={rgb, 255:red, 0; green, 0; blue, 0 }  ,fill opacity=1 ] (460.5,301) -- (469.5,301) -- (469.5,309.5) -- (460.5,309.5) -- cycle ;
\draw  [fill={rgb, 255:red, 0; green, 0; blue, 0 }  ,fill opacity=1 ] (410.5,351) -- (419.5,351) -- (419.5,359.5) -- (410.5,359.5) -- cycle ;
\draw  [fill={rgb, 255:red, 0; green, 0; blue, 0 }  ,fill opacity=1 ] (400.5,331) -- (409.5,331) -- (409.5,339.5) -- (400.5,339.5) -- cycle ;
\draw  [fill={rgb, 255:red, 0; green, 0; blue, 0 }  ,fill opacity=1 ] (420.5,221) -- (429.5,221) -- (429.5,229.5) -- (420.5,229.5) -- cycle ;
\draw  [fill={rgb, 255:red, 0; green, 0; blue, 0 }  ,fill opacity=1 ] (440.5,211) -- (449.5,211) -- (449.5,219.5) -- (440.5,219.5) -- cycle ;
\draw    (394,369.25) -- (553.5,365) ;
\draw    (415,370.25) -- (574,355) ;
\draw    (385,369.25) -- (543,342.5) ;
\draw    (404.75,370.25) -- (561.5,337) ;
\draw    (473.75,370.25) -- (627.5,325.5) ;
\draw    (495.75,370.25) -- (645.5,314.5) ;
\draw    (467.75,370.25) -- (615,305.5) ;
\draw    (482,370.25) -- (622,293) ;
\draw    (353.75,370.25) -- (489.25,285.5) ;
\draw    (378,370.25) -- (505.5,275) ;
\draw    (342.75,370.25) -- (461.5,263.5) ;
\draw    (362.75,370.25) -- (473.5,255.5) ;
\draw    (435.75,370.25) -- (535.5,244.5) ;
\draw    (427,370.25) -- (498,227.5) ;
\draw    (447.75,370.25) -- (493,217) ;
\draw    (73.75,370.25) -- (163,236.5) ;
\draw    (453.75,370.25) -- (543,236.5) ;
\draw    (427,370.25) -- (498,227.5) ;
\draw    (447.75,370.25) -- (493,217) ;

\draw (34,382) node [anchor=north west][inner sep=0.75pt]    {$a$};
\draw (6,354) node [anchor=north west][inner sep=0.75pt]    {$v$};

\end{tikzpicture}
\caption{An illustration of $\mathbb{K}_0$ for $K = C$, as in Figure \ref{figure1}, which gives rise to a non-sticky Kakeya set (left), and for a self-similar $K$ that gives rise to a sticky Kakeya set (right). For ease of interpretation, $K$ was reflected across the line $y=x$.}
\label{Fig-3}
\end{figure}
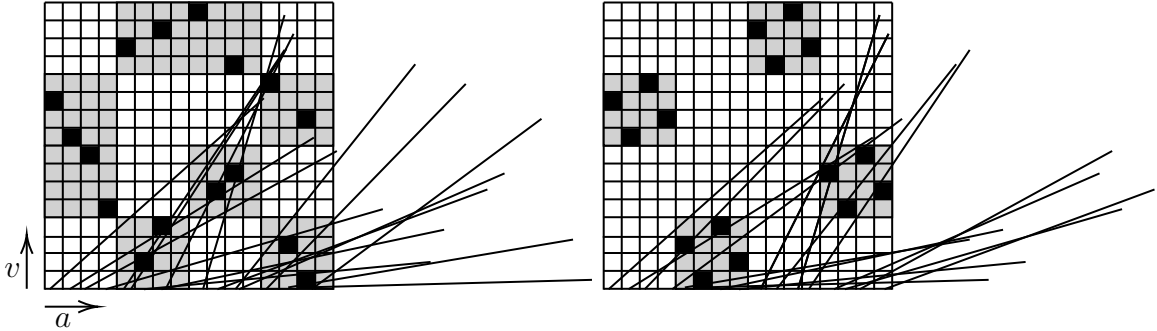

\section{Non-Trivial Constructions of Higher-Dimensional Kakeya sets}\label{section4}

In this section, we provide simple constructions of measure zero sticky and non-sticky Kakeya sets on $\R^d$ that are {\it not} given by a Cartesian product of a two-dimensional Kakeya set with $\R^{d-2}$. This appears to be unknown in the literature.  We also prove that both sets has Hausdorff dimension $d$.

\subsection{A Sticky Kakeya set on $\R^d$.} Instead of using the point-line duality presentation (\ref{eqL_va}), we construct the sticky Kakeya set by drawing lines between two Cantor sets as in \cite{Stein-real}, so that readers can visualize the construction easily (see Figure \ref{Fig-2}). Let $C_0$ be the self-similar Cantor set on $\R^1$ consisting of points whose $4$-adic expansions have digits $0,3$ only, i.e.,
$$
C_0 = \left\{\sum_{j=1}^{\infty}4^{-j}\varepsilon_j: \varepsilon_j\in\{0,3\}\right\}.
$$
On $\R^{d-1}$, let $C_{d-1} = C_0\times \cdots \times C_0$ ($d-1$ times).
 A crucial observation to proving our main theorem is that for all $t\in\R$, the Minkowski sum
    $$
    C_{d-1}+tC_{d-1} = (C_0+tC_0)\times\cdots\times (C_0+tC_0).
    $$

\begin{lem}\label{lemma4.1} The set $C_{d-1}+tC_{d-1}$ has the following properties:
      \begin{enumerate}
        \item  When $t = -1/2$, $C_0+tC_0 =[-\frac12,1]$ and  $C_{d-1}+tC_{d-1} = [-\frac12,1]^{d-1}$.
        \item $m^{d-1}(C_{d-1}+tC_{d-1}) = 0$, for almost all $t\in\R$. 
        \item $\hdim (C_{d-1}+tC_{d-1}) = d-1$, for almost all $t\in\R$. 
    \end{enumerate}
\end{lem}

\begin{proof}
    Property (1) follows from a direct calculation. For (2), it is known that $m^1(C_0+tC_0) = 0$, for almost all $t\in \R$ (see, e.g., \cite{Stein-real} or \cite{Mattila_2015}). By Fubini's theorem,   $m^{d-1}(C_{d-1}+tC_{d-1}) = 0$ follows.

    To see (3), it is well-known that $C_0+tC_0$ is the image of the orthogonal projection of the Cantor set $C_0\times C_0$ in $\R^2$ onto the subspace $y = tx$. By Marstrand's projection theorem (see, e.g., \cite[Section 2.3]{BKS2024}), for almost all $t\in \R$,  $\hdim(C_0+tC_0) = \hdim \pi_t(C_0\times C_0) = 1$. Since $\pdim(C_0+tC_0)$ is bounded below by $\hdim(C_0+tC_0)$, we have $\pdim(C_0+tC_0) = 1$, as well. Hence, the product inequality for Hausdorff dimension \eqref{hdimeq1} implies $\hdim(C_{d-1}+tC_{d-1}) = d-1$, for almost all $t\in\R$.
\end{proof}

Let $C = \frac12C_{d-1}\times \{0\}$ and $C' = C_{d-1}\times \{1\}$, and for each $a,b\in \R^d$, let $\ell(a,b)$ denote the line joining $a,b$. Define
$$
{\mathbb B} = \bigcup_{(a,b)\in C\times C'} \ell(a,b).
$$
\begin{theorem}
   A finite union of rotated copies of ${\mathbb B}$ is a compact Kakeya set. This Kakeya set is sticky, has Lebesgue measure zero, and has Hausdorff dimension $d$. 
\end{theorem}

\begin{proof}
The set of all directions in ${\mathbb B}$ is exactly equal to 
$$
C'-C = \left(C_{d-1}-\frac12 C_{d-1} \right)\times \{1\}.
 $$   
By Lemma \ref{lemma4.1} (1), the direction set covers a region in $\R^{d-1}\times \{1\}$ with non-empty interior, so the projection of this set onto the unit sphere $S^{d-1}$ has non-empty interior. Thus, a finite union of rotations of ${\mathbb B}$ will have its directions cover $S^{d-1}$, hence the union is a Kakeya set. 

To verify the set is sticky, we compute ${\mathcal A}_{\mathbb B}$ as 
$$
{\mathcal A}_{\mathbb B} = \bigcup_{a\in \frac{1}{2}C_{d-1}}  \left(C_{d-1}-a\right)\times \{a\}.
$$
This set is bi-Lipschitz equivalent to $C_{d-1}\times C_{d-1}$ via the bi-Lipschitz map $(x,y)\mapsto (y-2x,2x)$, and it has Hausdorff dimension $d-1$, since $\pdim(C_0) = \hdim(C_0) = 1/2$.

\begin{figure}
\centering
\tikzset{every picture/.style={line width=0.75pt}} 

\tikzset{every picture/.style={line width=0.75pt}} 

\begin{tikzpicture}[x=0.75pt,y=0.75pt,yscale=-1,xscale=1]

\draw  [fill={rgb, 255:red, 214; green, 214; blue, 214 }  ,fill opacity=1 ] (247.38,123.89) -- (415.15,200.84) -- (206.48,308.22) -- (38.71,231.27) -- cycle ;
\draw    (278.38,52.86) -- (96.94,146.21) ;
\draw    (314.74,69.53) -- (133.3,162.89) ;
\draw    (351.1,86.21) -- (169.66,179.56) ;
\draw    (196.8,59.45) -- (342.24,126.15) ;
\draw    (151.59,82.71) -- (297.02,149.42) ;
\draw    (106.37,105.98) -- (251.8,172.68) ;
\draw   (242.02,36.18) -- (387.44,102.88) -- (206.58,195.95) -- (61.16,129.25) -- cycle ;
\draw    (251.11,40.35) -- (69.67,133.7) ;
\draw    (260.2,44.52) -- (78.76,137.87) ;
\draw    (269.29,48.69) -- (87.85,142.04) ;
\draw    (287.47,57.03) -- (106.03,150.38) ;
\draw    (296.56,61.2) -- (115.12,154.55) ;
\draw    (305.65,65.36) -- (124.21,158.72) ;
\draw    (323.83,73.7) -- (142.39,167.06) ;
\draw    (332.92,77.87) -- (151.48,171.23) ;
\draw    (342.01,82.04) -- (160.57,175.4) ;
\draw    (360.19,90.38) -- (178.74,183.73) ;
\draw    (369.28,94.55) -- (187.83,187.9) ;
\draw    (378.37,98.72) -- (196.92,192.07) ;
\draw    (72.45,123.43) -- (217.88,190.13) ;
\draw    (83.76,117.61) -- (229.19,184.32) ;
\draw    (95.06,111.79) -- (240.49,178.5) ;
\draw    (117.67,100.16) -- (263.1,166.87) ;
\draw    (128.98,94.35) -- (274.41,161.05) ;
\draw    (140.28,88.53) -- (285.71,155.24) ;
\draw    (162.89,76.9) -- (308.32,143.6) ;
\draw    (174.2,71.08) -- (319.63,137.79) ;
\draw    (185.5,65.26) -- (330.93,131.97) ;
\draw    (208.11,53.63) -- (353.54,120.34) ;
\draw    (220.54,47.23) -- (365.98,113.94) ;
\draw    (230.72,42) -- (376.15,108.7) ;
\draw  [fill={rgb, 255:red, 0; green, 0; blue, 0 }  ,fill opacity=1 ] (72.23,124.43) -- (80.41,128.18) -- (70.24,133.41) -- (62.06,129.66) -- cycle ;
\draw  [fill={rgb, 255:red, 0; green, 0; blue, 0 }  ,fill opacity=1 ] (99.5,136.93) -- (107.68,140.69) -- (97.5,145.92) -- (89.33,142.17) -- cycle ;
\draw  [fill={rgb, 255:red, 0; green, 0; blue, 0 }  ,fill opacity=1 ] (133.41,119.48) -- (141.59,123.24) -- (131.42,128.47) -- (123.24,124.72) -- cycle ;
\draw  [fill={rgb, 255:red, 0; green, 0; blue, 0 }  ,fill opacity=1 ] (106.14,106.98) -- (114.32,110.73) -- (104.15,115.96) -- (95.97,112.21) -- cycle ;
\draw  [fill={rgb, 255:red, 0; green, 0; blue, 0 }  ,fill opacity=1 ] (207.89,54.63) -- (216.07,58.38) -- (205.89,63.62) -- (197.71,59.86) -- cycle ;
\draw  [fill={rgb, 255:red, 0; green, 0; blue, 0 }  ,fill opacity=1 ] (235.16,67.14) -- (243.34,70.89) -- (233.16,76.12) -- (224.98,72.37) -- cycle ;
\draw  [fill={rgb, 255:red, 0; green, 0; blue, 0 }  ,fill opacity=1 ] (242.93,36.6) -- (251.11,40.35) -- (240.94,45.58) -- (232.76,41.83) -- cycle ;
\draw  [fill={rgb, 255:red, 0; green, 0; blue, 0 }  ,fill opacity=1 ] (270.2,49.1) -- (278.38,52.86) -- (268.21,58.09) -- (260.03,54.34) -- cycle ;
\draw  [fill={rgb, 255:red, 0; green, 0; blue, 0 }  ,fill opacity=1 ] (352.01,86.63) -- (360.19,90.38) -- (350.01,95.61) -- (341.83,91.86) -- cycle ;
\draw  [fill={rgb, 255:red, 0; green, 0; blue, 0 }  ,fill opacity=1 ] (379.27,99.13) -- (387.45,102.89) -- (377.28,108.12) -- (369.1,104.37) -- cycle ;
\draw  [fill={rgb, 255:red, 0; green, 0; blue, 0 }  ,fill opacity=1 ] (318.09,104.08) -- (326.27,107.83) -- (316.1,113.06) -- (307.92,109.31) -- cycle ;
\draw  [fill={rgb, 255:red, 0; green, 0; blue, 0 }  ,fill opacity=1 ] (345.36,116.58) -- (353.54,120.34) -- (343.37,125.57) -- (335.19,121.82) -- cycle ;
\draw  [fill={rgb, 255:red, 0; green, 0; blue, 0 }  ,fill opacity=1 ] (216.35,156.42) -- (224.53,160.18) -- (214.35,165.41) -- (206.18,161.66) -- cycle ;
\draw  [fill={rgb, 255:red, 0; green, 0; blue, 0 }  ,fill opacity=1 ] (243.62,168.93) -- (251.8,172.68) -- (241.62,177.92) -- (233.44,174.17) -- cycle ;
\draw  [fill={rgb, 255:red, 0; green, 0; blue, 0 }  ,fill opacity=1 ] (182.43,173.87) -- (190.61,177.63) -- (180.44,182.86) -- (172.26,179.11) -- cycle ;
\draw  [fill={rgb, 255:red, 0; green, 0; blue, 0 }  ,fill opacity=1 ] (209.7,186.38) -- (217.88,190.13) -- (207.71,195.37) -- (199.53,191.62) -- cycle ;
\draw    (278.38,212.86) -- (96.94,306.21) ;
\draw    (314.74,229.53) -- (133.3,322.89) ;
\draw    (351.1,246.21) -- (169.66,339.56) ;
\draw    (196.8,219.45) -- (342.24,286.15) ;
\draw    (151.59,242.71) -- (297.02,309.42) ;
\draw    (106.37,265.98) -- (251.8,332.68) ;
\draw   (242.02,196.18) -- (387.44,262.88) -- (206.58,355.95) -- (61.16,289.25) -- cycle ;
\draw    (251.11,200.35) -- (69.67,293.7) ;
\draw    (260.2,204.52) -- (78.76,297.87) ;
\draw    (269.29,208.69) -- (87.85,302.04) ;
\draw    (287.47,217.03) -- (106.03,310.38) ;
\draw    (296.56,221.2) -- (115.12,314.55) ;
\draw    (305.65,225.36) -- (124.21,318.72) ;
\draw    (323.83,233.7) -- (142.39,327.06) ;
\draw    (332.92,237.87) -- (151.48,331.23) ;
\draw    (342.01,242.04) -- (160.57,335.4) ;
\draw    (360.19,250.38) -- (178.74,343.73) ;
\draw    (369.28,254.55) -- (187.83,347.9) ;
\draw    (378.37,258.72) -- (196.92,352.07) ;
\draw    (72.45,283.43) -- (217.88,350.13) ;
\draw    (83.76,277.61) -- (229.19,344.32) ;
\draw    (95.06,271.79) -- (240.49,338.5) ;
\draw    (117.67,260.16) -- (263.1,326.87) ;
\draw    (128.98,254.35) -- (274.41,321.05) ;
\draw    (140.28,248.53) -- (285.71,315.24) ;
\draw    (162.89,236.9) -- (308.32,303.6) ;
\draw    (174.2,231.08) -- (319.63,297.79) ;
\draw    (185.5,225.26) -- (330.93,291.97) ;
\draw    (208.11,213.63) -- (353.54,280.34) ;
\draw    (220.54,207.23) -- (365.98,273.94) ;
\draw    (230.72,202) -- (376.15,268.7) ;
\draw  [fill={rgb, 255:red, 0; green, 0; blue, 0 }  ,fill opacity=1 ] (72.23,284.43) -- (80.41,288.18) -- (70.24,293.41) -- (62.06,289.66) -- cycle ;
\draw  [fill={rgb, 255:red, 0; green, 0; blue, 0 }  ,fill opacity=1 ] (99.5,296.93) -- (107.68,300.69) -- (97.5,305.92) -- (89.33,302.17) -- cycle ;
\draw  [fill={rgb, 255:red, 0; green, 0; blue, 0 }  ,fill opacity=1 ] (133.41,279.48) -- (141.59,283.24) -- (131.42,288.47) -- (123.24,284.72) -- cycle ;
\draw  [fill={rgb, 255:red, 0; green, 0; blue, 0 }  ,fill opacity=1 ] (106.14,266.98) -- (114.32,270.73) -- (104.15,275.96) -- (95.97,272.21) -- cycle ;
\draw  [fill={rgb, 255:red, 0; green, 0; blue, 0 }  ,fill opacity=1 ] (207.89,214.63) -- (216.07,218.38) -- (205.89,223.62) -- (197.71,219.86) -- cycle ;
\draw  [fill={rgb, 255:red, 0; green, 0; blue, 0 }  ,fill opacity=1 ] (235.16,227.14) -- (243.34,230.89) -- (233.16,236.12) -- (224.98,232.37) -- cycle ;
\draw  [fill={rgb, 255:red, 0; green, 0; blue, 0 }  ,fill opacity=1 ] (242.93,196.6) -- (251.11,200.35) -- (240.94,205.58) -- (232.76,201.83) -- cycle ;
\draw  [fill={rgb, 255:red, 0; green, 0; blue, 0 }  ,fill opacity=1 ] (270.2,209.1) -- (278.38,212.86) -- (268.21,218.09) -- (260.03,214.34) -- cycle ;
\draw  [fill={rgb, 255:red, 0; green, 0; blue, 0 }  ,fill opacity=1 ] (352.01,246.63) -- (360.19,250.38) -- (350.01,255.61) -- (341.83,251.86) -- cycle ;
\draw  [fill={rgb, 255:red, 0; green, 0; blue, 0 }  ,fill opacity=1 ] (379.27,259.13) -- (387.45,262.89) -- (377.28,268.12) -- (369.1,264.37) -- cycle ;
\draw  [fill={rgb, 255:red, 0; green, 0; blue, 0 }  ,fill opacity=1 ] (318.09,264.08) -- (326.27,267.83) -- (316.1,273.06) -- (307.92,269.31) -- cycle ;
\draw  [fill={rgb, 255:red, 0; green, 0; blue, 0 }  ,fill opacity=1 ] (345.36,276.58) -- (353.54,280.34) -- (343.37,285.57) -- (335.19,281.82) -- cycle ;
\draw  [fill={rgb, 255:red, 0; green, 0; blue, 0 }  ,fill opacity=1 ] (216.35,316.42) -- (224.53,320.18) -- (214.35,325.41) -- (206.18,321.66) -- cycle ;
\draw  [fill={rgb, 255:red, 0; green, 0; blue, 0 }  ,fill opacity=1 ] (243.62,328.93) -- (251.8,332.68) -- (241.62,337.92) -- (233.44,334.17) -- cycle ;
\draw  [fill={rgb, 255:red, 0; green, 0; blue, 0 }  ,fill opacity=1 ] (182.43,333.87) -- (190.61,337.63) -- (180.44,342.86) -- (172.26,339.11) -- cycle ;
\draw  [fill={rgb, 255:red, 0; green, 0; blue, 0 }  ,fill opacity=1 ] (209.7,346.38) -- (217.88,350.13) -- (207.71,355.37) -- (199.53,351.62) -- cycle ;
\draw    (241.94,41.09) -- (241.94,201.09) ;
\draw    (241.94,41.09) -- (269.2,213.6) ;
\draw    (241.94,41.09) -- (234.16,231.63) ;
\draw    (241.94,41.09) -- (206.89,219.12) ;
\draw    (241.94,41.09) -- (351.01,251.12) ;
\draw    (241.94,41.09) -- (378.28,263.63) ;
\draw    (241.94,41.09) -- (344.36,281.08) ;
\draw    (242.93,36.6) -- (317.1,268.57) ;
\draw    (241.94,41.09) -- (242.62,333.43) ;
\draw    (241.94,41.09) -- (215.35,320.92) ;
\draw    (241.94,41.09) -- (208.71,350.88) ;
\draw    (241.94,41.09) -- (181.44,338.37) ;
\draw    (241.94,41.09) -- (132.42,283.98) ;
\draw    (242.93,36.6) -- (105.15,271.47) ;
\draw    (241.94,41.09) -- (71.23,288.92) ;
\draw    (241.94,41.09) -- (98.5,301.43) ;
\draw    (267,52.05) -- (241.94,201.09) ;
\draw    (269.2,53.6) -- (267,212.05) ;
\draw    (269.2,53.6) -- (234.16,231.63) ;
\draw    (269.2,53.6) -- (206.89,219.12) ;
\draw    (267,52.05) -- (351.01,251.12) ;
\draw    (269.2,53.6) -- (378.28,263.63) ;
\draw    (268.21,58.09) -- (343.37,285.57) ;
\draw    (269.2,53.6) -- (317.1,268.57) ;
\draw    (269.2,53.6) -- (242.62,333.43) ;
\draw    (267,52.05) -- (215.35,320.92) ;
\draw    (269.2,53.6) -- (208.71,350.88) ;
\draw    (269.2,53.6) -- (181.44,338.37) ;
\draw    (267,52.05) -- (132.42,283.98) ;
\draw    (269.2,53.6) -- (98.5,301.43) ;
\draw    (269.2,53.6) -- (71.23,288.92) ;
\draw    (269.2,53.6) -- (105.15,271.47) ;
\draw    (204,59.05) -- (241.94,201.09) ;
\draw    (206.89,59.12) -- (269.2,213.6) ;
\draw    (207.89,54.63) -- (209,217.05) ;
\draw    (206.89,59.12) -- (235.16,227.14) ;
\draw    (206.89,59.12) -- (351.01,251.12) ;
\draw    (206.89,59.12) -- (377.28,268.12) ;
\draw    (206.89,59.12) -- (344.36,281.08) ;
\draw    (207.89,54.63) -- (314,269.05) ;
\draw    (206.89,59.12) -- (242.62,333.43) ;
\draw    (206.89,59.12) -- (215.35,320.92) ;
\draw    (209.5,57.05) -- (208.71,350.88) ;
\draw    (209.5,57.05) -- (181.44,338.37) ;
\draw    (206.89,59.12) -- (132.42,283.98) ;
\draw    (209.5,57.05) -- (95.5,302.55) ;
\draw    (206.89,59.12) -- (106,274.05) ;
\draw    (209.5,57.05) -- (71.23,288.92) ;
\draw    (234.16,71.63) -- (241.94,201.09) ;
\draw    (234.16,71.63) -- (269.2,213.6) ;
\draw    (234.16,71.63) -- (206.89,219.12) ;
\draw    (234.16,71.63) -- (237,232.55) ;
\draw    (235.16,67.14) -- (347.5,250.05) ;
\draw    (234.16,71.63) -- (375.5,264.05) ;
\draw    (234.16,71.63) -- (317.1,268.57) ;
\draw    (237.5,70.55) -- (344.36,281.08) ;
\draw    (234.16,71.63) -- (217.5,319.05) ;
\draw    (237.5,70.55) -- (242.62,333.43) ;
\draw    (234.16,71.63) -- (208.71,350.88) ;
\draw    (231.5,69.05) -- (177.5,339.05) ;
\draw    (237.5,70.55) -- (107.5,270.05) ;
\draw    (237.5,70.55) -- (132.42,283.98) ;
\draw    (237.5,70.55) -- (78,287.05) ;
\draw    (235.16,67.14) -- (95.5,302.55) ;

\draw (434,186.05) node [anchor=north west][inner sep=0.75pt]    {$x_{d} =\lambda $};
\draw (434,253.05) node [anchor=north west][inner sep=0.75pt]    {$x_{d} =0$};
\draw (434,91.05) node [anchor=north west][inner sep=0.75pt]    {$x_{d} =1$};
\end{tikzpicture}
\caption{An illustration of ${\mathbb B}$ for $d=3$ and the slice $x_d = \lambda$. The copy of $C_{d-1}$ at $x_d = 0$ has been scaled up by a factor of 2, and only the lines emanating from one corner of the copy of $C_{d-1}$ at $x_d=1$ are shown.}
\label{Fig-2}
\end{figure}
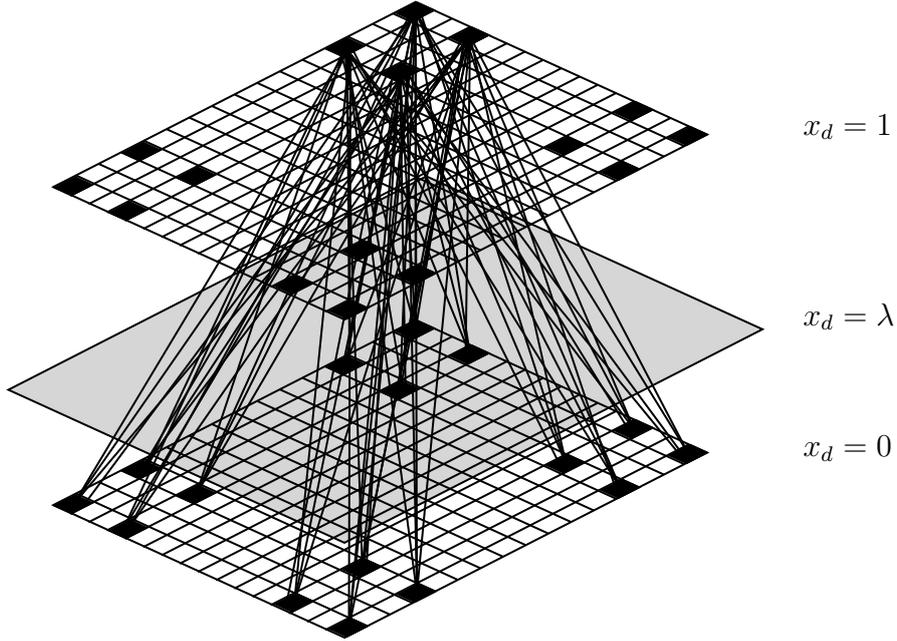

To show that the Kakeya set has Lebesgue measure zero and Hausdorff dimension $d$, we compute the slice ${\mathbb B}\cap \{x_d = \lambda\}$ for $\lambda\in(0,1)$. Indeed, 
$$
{\mathbb B}\cap \{x_d = \lambda\} = \left\{((1-\lambda)a+\lambda b,\lambda): (a,b)\in C\times C'\right\}.
$$
Letting $t = \frac{1-\lambda}{\lambda}$ and factoring out $\lambda$, the slice is a rescaling of $C_{d-1}+tC_{d-1}$. By Lemma \ref{lemma4.1} (2) and (3), the slice has Lebesgue measure zero and Hausdorff dimension $d-1$, for almost all $t\in\R$ and hence $\lambda\in(0,1)$. By Fubini's theorem, ${\mathbb B}$ has measure zero. Finally, by Marstrand's slicing theorem (see, e.g., \cite[Section 2.4]{BKS2024}), 
$$
\hdim({\mathbb B}) \ge 1+ \hdim({\mathbb B}\cap \{x_d = \lambda\}) = 1+(d-1) = d,
$$
for almost all $\lambda$.
\end{proof}

\subsection{Non-sticky Kakeya sets on $\R^d$.} We now provide a construction of a non-trivial non-sticky Kakeya set in $\R^{d}$. Let $C$ be the Cantor set on $\R^2$ from Proposition \ref{prop3.3} that leads to the non-sticky Kakeya set in $\R^2$. Unlike the sticky construction above, $C$ cannot be decomposed as the Cartesian product of two identical Cantor sets. Modifying our technique, we use $C$ as the basic set and consider the Cartesian product $C^{d-1} \subset \R^{2d-2}$ in the parameter space of lines. This time we view
$$
({\bf v}, {\bf a}) = (v_1,a_1) \times \cdots \times (v_{d_1}, a_{d-1}) \in C^{d-1}.
$$
\begin{theorem}
A finite union of rotated copies of
$$
\mathbb{K}_0 = \bigcup_{({\bf v}, {\bf a})\in C^{d-1}} L_{{\bf v}, {\bf a}}
$$
is a non-sticky Kakeya set of Lebesgue measure zero and Haudorff dimension $d$.
\end{theorem}

\begin{proof}
By Proposition \ref{prop_proj}, $\pi_0(C)$ has non-empty interior in $\R^1$, so under the projection
$$
\widetilde{\pi}_0(v_1,a_1, \dots, v_{d-1},a_{d-1}) = (v_1, \dots, v_{d-1}),
$$
the interior
$$
[\widetilde{\pi}_0(C^{d-1})]^{\mathrm{o}} = \pi_0(C)^{\mathrm{o}}\times\cdots\times \pi_0(C)^{\mathrm{o}} \neq \varnothing.
$$
By Proposition \ref{prop0}, modified to use $\Tilde{\pi}_0$ in place of $\pi_0$, a finite union of rotated copies of $\mathbb{K}_0$ is a Kakeya set in $\R^{d}$. By the product inequality for packing dimension \eqref{hdimeq2} and Proposition \ref{prop3.3} (2),
$$d - 1 < 1 + (d-2)\cdot\pdim(C) \le \pdim(C^{d-1}) \leq (d-1)\cdot\pdim(C),$$
so since $\mathcal{A}_{\mathbb{K}_0} = C^{d-1}$,
the Kakeya set is non-sticky.

For $i= 1, \dots, d-1$, define $e_i = (0, \dots, 0, t, 1, 0, \dots, 0)$, where $t$ appears in the $(2i-2)^{\mathrm{nd}}$ entry and 1 appears in the $(2i-1)^{\mathrm{st}}$ entry.  Note that 
$$
 L_{{\bf v}, {\bf a}} = \{(tv_1+a_1,\cdots, tv_{d-1}+a_{d-1},t): t\in[0,1]\}
$$Then the slice of $\mathbb{K}_0$ at $z = t$ is exactly the following $d-1$ product:
$$
W_t = \pi_t(C)\times\cdots\times \pi_t(C).
$$
Since $m^1(\pi_t(C)) = 0$, for almost all $t \in \R$, by Fubini's theorem, $W_t$ has $(d-1)$-dimensional Lebesgue measure zero,
for almost all $t \in \R$. Then by Fubini's theorem again, $\mathbb{K}_0$ has measure zero, so the Kakeya set also has measure zero.

Finally, we show that this Kakeya set has full Hausdorff dimension. As $\hdim C = 1$, $\hdim(\pi_t (C)) = 1$, for almost every $t$, by Marstrand's projection theorem.   By the product inequality for Hausdorff dimension (\ref{hdimeq1}), $\hdim W_t\ge d-1$. By Marstrand's slicing theorem, $\hdim({\mathbb K}_0) = d$, so the Kakeya set has Hausdorff dimension $d$.
\end{proof}

We remark that we can always produce a Kakeya set in $\R^d$ using the above procedure when we can construct $C$ in $\R^2$ satisfying Proposition \ref{prop-K} with (1), (3) and (4). It will be non-sticky if we also have (2). 

\medskip

\bibliographystyle{plainurl}
\bibliography{refs_plan}

@book{Mattila_2015,
  place={Cambridge},
  series={Cambridge Studies in Advanced Mathematics},
  title={Fourier Analysis and Hausdorff Dimension},
  publisher={Cambridge University Press},
  author={Mattila, Pertti},
  year={2015},
  collection={Cambridge Studies in Advanced Mathematics}
}

@article {FLR02,
    AUTHOR = {Fan, Ai-Hua and Lau, Ka-Sing and Rao, Hui},
     TITLE = {Relationships between different dimensions of a measure},
   JOURNAL = {Monatsh. Math.},
  FJOURNAL = {Monatshefte f\"ur Mathematik},
    VOLUME = {135},
      YEAR = {2002},
    NUMBER = {3},
     PAGES = {191--201},
      ISSN = {0026-9255,1436-5081},
   MRCLASS = {28A80},
  MRNUMBER = {1897575},
MRREVIEWER = {M.\ K.\ Nayak},
       DOI = {10.1007/s006050200016},
       URL = {https://doi.org/10.1007/s006050200016},
}

@book {BKS2024,
    AUTHOR = {B\'ar\'any, Bal\'azs and Simon, K\'aroly and Solomyak, Boris},
     TITLE = {Self-similar and self-affine sets and measures},
    SERIES = {Mathematical Surveys and Monographs},
    VOLUME = {276},
 PUBLISHER = {American Mathematical Society, Providence, RI},
      YEAR = {[2023] },
     PAGES = {xii+451},
      ISBN = {[9781470470463]; [9781470475505]},
   MRCLASS = {28-02 (28A75 28A78 28A80 37D35 42A38)},
  MRNUMBER = {4661364},
MRREVIEWER = {J\"org\ Neunh\"auserer},
       DOI = {10.1090/surv/276},
       URL = {https://doi.org/10.1090/surv/276},
}

@article {Davis71,
    AUTHOR = {Davies, Roy O.},
     TITLE = {Some remarks on the {K}akeya problem},
   JOURNAL = {Proc. Cambridge Philos. Soc.},
  FJOURNAL = {Proceedings of the Cambridge Philosophical Society},
    VOLUME = {69},
      YEAR = {1971},
     PAGES = {417--421},
      ISSN = {0008-1981},
   MRCLASS = {28.80},
  MRNUMBER = {272988},
MRREVIEWER = {G.\ Freilich},
       DOI = {10.1017/s0305004100046867},
       URL = {https://doi.org/10.1017/s0305004100046867},
}

@article{wang2022stickykakeyasetssticky,
      title={Sticky {K}akeya sets and the sticky Kakeya conjecture}, 
      author={Hong Wang and Joshua Zahl},
      year={2022},
      eprint={2210.09581},
      archivePrefix={arXiv},
      primaryClass={math.CA},
      url={}, 
}

@misc{wang2025assouaddimensionkakeyasets,
      title={The {A}ssouad dimension of {K}akeya sets in $\mathbb{R}^3$}, 
      author={Hong Wang and Joshua Zahl},
      year={2025},
      eprint={2401.12337},
      archivePrefix={arXiv},
      primaryClass={math.CA},
      url={}, 
}

@misc{wang2025volumeestimatesunionsconvex,
      title={Volume estimates for unions of convex sets, and the {K}akeya set conjecture in three dimensions}, 
      author={Hong Wang and Joshua Zahl},
      year={2025},
      eprint={2502.17655},
      archivePrefix={arXiv},
      primaryClass={math.CA},
      url={}, 
}

@misc{choudhuri2024improvedboundhausdorffdimension,
      title={An improved bound on the Hausdorff dimension of sticky Kakeya sets in $\mathbb{R}^4$}, 
      author={Mukul Rai Choudhuri},
      year={2024},
      eprint={2410.23579},
      archivePrefix={arXiv},
      primaryClass={math.CA},
      url={}, 
}

@article {KLT2000,
    AUTHOR = {Katz, Nets Hawk and {\L}aba, Izabella and Tao, Terence},
     TITLE = {An improved bound on the {M}inkowski dimension of
              {B}esicovitch sets in {${\bf R}^3$}},
   JOURNAL = {Ann. of Math. (2)},
  FJOURNAL = {Annals of Mathematics. Second Series},
    VOLUME = {152},
      YEAR = {2000},
    NUMBER = {2},
     PAGES = {383--446},
      ISSN = {0003-486X,1939-8980},
   MRCLASS = {28A75 (28A78 42B25 46E30)},
  MRNUMBER = {1804528},
MRREVIEWER = {Lubo\v s\ Pick},
       DOI = {10.2307/2661389},
       URL = {https://doi.org/10.2307/2661389},
}

@article {Taoblog,
AUTHOR = {Tao, Terence},
 TITLE = {Stickiness, Graininess, Planiness, and a Sum-Product approach to the Kakeya Problem, blog post},    
  YEAR = {2014},
URL = {https://terrytao.wordpress.com/2014/05/07/stickiness-graininess-planiness-and-a-sum-product-approach-to-the-kakeya-problem},
}

@book {Stein-real,
    AUTHOR = {Stein, Elias M. and Shakarchi, Rami},
     TITLE = {Real analysis},
    SERIES = {Princeton Lectures in Analysis},
    VOLUME = {3},
      NOTE = {Measure theory, integration, and Hilbert spaces},
 PUBLISHER = {Princeton University Press, Princeton, NJ},
      YEAR = {2005},
     PAGES = {xx+402},
      ISBN = {0-691-11386-6},
   MRCLASS = {28B15 (00-01 26-01 26A03 26A09 28-01 46-01 47-01)},
  MRNUMBER = {2129625},
MRREVIEWER = {Steven\ George\ Krantz},
}

@article {Tricot82,
    AUTHOR = {Tricot, Jr., Claude},
     TITLE = {Two definitions of fractional dimension},
   JOURNAL = {Math. Proc. Cambridge Philos. Soc.},
  FJOURNAL = {Mathematical Proceedings of the Cambridge Philosophical
              Society},
    VOLUME = {91},
      YEAR = {1982},
    NUMBER = {1},
     PAGES = {57--74},
      ISSN = {0305-0041,1469-8064},
   MRCLASS = {28A75 (28A10)},
  MRNUMBER = {633256},
MRREVIEWER = {G.\ Freilich},
       DOI = {10.1017/S0305004100059119},
       URL = {https://doi.org/10.1017/S0305004100059119},
}

@book{falconer1985geometry,
  title={The geometry of fractal sets},
  author={Falconer, Kenneth J},
  number={85},
  year={1985},
  publisher={Cambridge university press}
}

@book {W03,
    AUTHOR = {Wolff, Thomas H.},
     TITLE = {Lectures on harmonic analysis},
    SERIES = {University Lecture Series},
    VOLUME = {29},
    EDITOR = {\L aba, Izabella and Shubin, Carol},
      NOTE = {With a foreword by Charles Fefferman and a preface by Izabella
              \L aba},
 PUBLISHER = {American Mathematical Society, Providence, RI},
      YEAR = {2003},
     PAGES = {x+137},
      ISBN = {0-8218-3449-5},
   MRCLASS = {42-02 (42B15)},
  MRNUMBER = {2003254},
MRREVIEWER = {Steven\ George\ Krantz},
       DOI = {10.1090/ulect/029},
       URL = {https://doi.org/10.1090/ulect/029},
}

@article {Guth2016,
    AUTHOR = {Guth, Larry},
     TITLE = {Degree reduction and graininess for {K}akeya-type sets in
              {$\Bbb{R}^3$}},
   JOURNAL = {Rev. Mat. Iberoam.},
  FJOURNAL = {Revista Matem\'atica Iberoamericana},
    VOLUME = {32},
      YEAR = {2016},
    NUMBER = {2},
     PAGES = {447--494},
      ISSN = {0213-2230,2235-0616},
   MRCLASS = {42B20 (52A35)},
  MRNUMBER = {3512423},
MRREVIEWER = {Shaoming\ Guo},
       DOI = {10.4171/RMI/891},
       URL = {https://doi.org/10.4171/RMI/891},
}

@article {BCT2006,
    AUTHOR = {Bennett, Jonathan and Carbery, Anthony and Tao, Terence},
     TITLE = {On the multilinear restriction and {K}akeya conjectures},
   JOURNAL = {Acta Math.},
  FJOURNAL = {Acta Mathematica},
    VOLUME = {196},
      YEAR = {2006},
    NUMBER = {2},
     PAGES = {261--302},
      ISSN = {0001-5962,1871-2509},
   MRCLASS = {42B20 (44A15 45P05 46E35)},
  MRNUMBER = {2275834},
MRREVIEWER = {Loukas\ Grafakos},
       DOI = {10.1007/s11511-006-0006-4},
       URL = {https://doi.org/10.1007/s11511-006-0006-4},
}

@article {KZ2019,
    AUTHOR = {Katz, Nets Hawk and Zahl, Joshua},
     TITLE = {An improved bound on the {H}ausdorff dimension of
              {B}esicovitch sets in {$\Bbb{R}^3$}},
   JOURNAL = {J. Amer. Math. Soc.},
  FJOURNAL = {Journal of the American Mathematical Society},
    VOLUME = {32},
      YEAR = {2019},
    NUMBER = {1},
     PAGES = {195--259},
      ISSN = {0894-0347,1088-6834},
   MRCLASS = {42B25},
  MRNUMBER = {3868003},
MRREVIEWER = {Israel\ P.\ Rivera-R\'ios},
       DOI = {10.1090/jams/907},
       URL = {https://doi.org/10.1090/jams/907},
}

\end{document}